\documentclass[12pt, reqno]{amsart}
\makeatletter
\@namedef{subjclassname@1991}{$\mathrm{1991}$ Mathematics Subject Classification}
\@namedef{subjclassname@2000}{$\mathrm{2000}$ Mathematics Subject Classification}
\@namedef{subjclassname@2010}{$\mathrm{2010}$ Mathematics Subject Classification}
\@namedef{subjclassname@2020}{$\mathrm{2020}$ Mathematics Subject Classification}
\makeatother
\usepackage{amsmath,amsthm, amscd, amsfonts, amssymb, graphicx, color}
\usepackage[bookmarksnumbered, colorlinks, plainpages,linkcolor=blue,urlcolor=blue,citecolor=blue]{hyperref}
\newtheorem{thm}{Theorem}[section]
\newtheorem{cor}[thm]{Corollary}
\newtheorem{lem}[thm]{Lemma}

\newtheorem{defn}[thm]{Definition}

\newtheorem{exam}[thm]{Example}
\numberwithin{equation}{section}

\begin{document}

\title{New properties of weighted generalized core-EP inverse in Banach algebras}

\author{Huanyin Chen}
\author{Marjan Sheibani}
\address{School of Big Data, Fuzhou University of International Studies and Trade, Fuzhou 350202, China}
\email{<huanyinchenfz@163.com>}
\address{Farzanegan Campus, Semnan University, Semnan, Iran}
\email{<m.sheibani@semnan.ac.ir>}

\subjclass[2020]{15A09, 16U99, 46H05.} \keywords{core-EP inverse; weighted core-EP inverse; generalized Drazin inverse;
weighted g-Drazin inverse; generalized weighted core-EP inverse; Banach algebra.}

\begin{abstract} We characterize the generalized weighted core-EP inverse via the canonical decomposition, utilizing a weighted core-EP invertible element and a quasinilpotent. We then offer a polar-like characterization for the generalized weighted core-EP invertible element. The representations of the generalized weighted core-EP inverse by leveraging the weighted generalized Drazin inverse are thereby presented. These lead to new properties for the weighted core-EP inverse.
\end{abstract}

\maketitle

\section{Introduction}

A Banach algebra is called a Banach *-algebra if there exists an involution $*: x\to x^*$ satisfying $(x+y)^*=x^*+y^*, (\lambda x)^*=\overline{\lambda} x^*, (xy)^*=y^*x^*, (x^*)^*=x$. Let ${\Bbb C}^{n\times n}$ be the Banach *-algebra of all $n\times n$ complex matrices with conjugate transpose $*$ and
$\mathcal{R}(X)$ represent the range space of a complex matrix $X$. In 2010,  Baksalary and Trenkler introduced the core inverse for a complex matrix (~\cite{BT}).
 A matrix $A\in C^{n\times n}$ has core inverse $X$ if and only if $AX=P_A, \mathcal{R}(X)\subseteq \mathcal{R}(A)$, where $P_A$ is the orthogonal projection on $\mathcal{R}(A)$. Such $X$ is unique, and we denote it by $A^{\tiny\textcircled{\#}}$.

In 2014, Prasad and Mohana extended core inverse and introduced core-EP inverse for a complex matrix (see~\cite{P}). A matrix $A\in C^{n\times n}$ has core-EP inverse $X$ if and only if $$XAX=X, \mathcal{R}(X)=\mathcal{R}(X^*)=\mathcal{R}(A^k),$$ where $k=ind(A)$ is the Drazin index of $A$. Such $X$ is unique, and we denote it by $A^{\tiny\textcircled{\dag}}$.

In 2020, Gao et al. extended the concept of the core-EP inverse and introduced the notion of weighted core-EP inverse for a complex matrix (see~\cite{G2}).
Let $A, W\in {\Bbb C}^{n\times n}$ and $k=max\{ ind(AW),ind(WA)\}$. The weighted core-EP inverse of $A$ is the unique solution to the system: $$WAWX=(WA)^k[(WA)^k]^{\dag}, \mathcal{R}(X)\subseteq \mathcal{R}((AW)^k),$$ and we denote such $X$ by $A^{\tiny\textcircled{\dag},W}$.

Then Mosi\'c introduced and studied weighted core-EP inverse for a bounded linear operator between two Hilbert spaces as a generalization of the weighted core-EP inverse of a matrix (see~\cite{M2}). In 2021, Mosi\'c further extended the weighted core-EP inverse of bounded linear operators on Hilbert spaces to elements of a $C^*$-algebra and the weighted core-EP inverse in a $C^*$-algebra was characterized by means of range projections (see~\cite{M3}).

Core, core-EP and weighted core-EP inverses are extensively studied by many authors, we refer the reader to~\cite{CZ,F,G1,M,MM,M4,R,X}.
 Recently, Chen and Sheibani introduced and studied generalized weighted core-EP inverse for an element in a Banach *-algebra (see~\cite{C}).

\begin{defn} An element $a\in \mathcal{A}$ has $w$-weighted generalized core-EP if there exist $x\in \mathcal{A}$ such that
$$a(wx)^2=x, (wawx)^*=wawx, \lim\limits_{n\to \infty}||(aw)^n-(xw)(aw)^{n+1}||^{\frac{1}{n}}=0.$$\end{defn}

The preceding $x$ is unique if exists, and we denote it by $a^{\tiny\textcircled{d},w}$. The generalized weighted core-EP inverse is a natural generalizations of those generalized inverses mentioned above. Many properties of the known generalized inverses are thereby extended to wider cases.

In this paper, we present new properties of generalized weighted core-EP inverse in a Banach *-algebra.

In Section 2, we introduce $w$-weighted core inverse for an element in a Banach *-algebra.

\begin{defn} An element $a\in \mathcal{A}$ has $w$-weighted core inverse if there exist $x\in \mathcal{A}$ such that
$$a(wx)^2=x, (wawx)^*=wawx, xw(aw)^2=aw.$$\end{defn} If such $x$ is unique if it exits, and we denote it by $a^{\tiny\textcircled{\#},w}$.
Many elementary properties of $w$-weighted core inverse are investigated.

In Section 3, we characterize weighted generalized core-EP inverse by using $w$-weighted core inverse and quasinilpotent in a Banach *-algebra.
We prove that $a\in \mathcal{A}^{\tiny\textcircled{d},w}$ if and only if
there exist $z,y\in \mathcal{A}$ such that $$a=z+y, (wz)^*(wy)=ywz=0, z\in \mathcal{A}^{\tiny\textcircled{\#},w}, y\in \mathcal{A}_w^{qnil}.$$
Here, $\mathcal{A}_w^{qnil}=\{ x\in \mathcal{A}\mid ~ \lim\limits_{n\to \infty}||(yw)^n||^{\frac{1}{n}}=0$. Surprisingly, we observe that
the preceding condition $(wz)^*(wy)=0$ can be dropped, which add some new properties for the weighted core-EP inverse.
The polar-like characterization of the generalized weighted core-EP inverse is established.

An element $a\in \mathcal{A}$ has weighted g-Drazin inverse $x$ if there exists unique $x\in \mathcal{A}$ such that
$$awx=xwa, xwawx=x ~\mbox{and}~ a-awxwa\in \mathcal{A}^{qnil}.$$ We denote $x$ by $a^{d,w}$ (see~\cite{M1}).
In Section 4, we characterize the generalized weighted core-EP inverse using the weighted g-Drazin inverse.
Its representations by using the weighted core inverse is then established.

Throughout the paper, all Banach *-algebras are complex with an identity. $\mathcal{A}^{d,w}, \mathcal{A}^{\#,w}, \mathcal{A}^{\tiny\textcircled{\#},w}$, and $\mathcal{A}^{\tiny\textcircled{d},w}$ denote the sets of all weighted g-Drazin, weighted group, weighted core and generalized weighted core-EP invertible elements in $\mathcal{A}$, respectively. We use $^0(a)$ to stand for the left annihilator $\{ x\in \mathcal{A}~\mid~xa=0\}$ of $a$.

\section{$w$-weighted core inverse}

The aim of this section is to elucidate the fundamental properties of the $w$-weighted core inverse. We commence by providing some general characterizations of the $w$-weighted core inverse within the context of a Banach algebra.

\begin{thm} Let $a,x\in \mathcal{A}$. Then the following are equivalent:\end{thm}
\begin{enumerate}
\item [(1)] $a^{\tiny\textcircled{\#},w}=x$.
\item [(2)] $xwawx=x, (wawx)^*=wawx, x\mathcal{A}=(aw)\mathcal{A}, x^*\mathcal{A}=(waw)\mathcal{A}$.
\item [(3)] $xwawx=x, (wawx)^*=wawx, ^0(x)=^0(aw), ^0(x^*)=^0(waw)$.
\item [(4)] $xwawx=x, (wawx)^*=wawx, ^0(x)=^0(aw), ^0(x^*)\subseteq ^0(waw)$.
\end{enumerate}
\begin{proof} $(1)\Rightarrow (2)$ $x=a(wx)^2=(aw)xwx\in (aw)\mathcal{A}$. $aw=xw(aw)^2\in x\mathcal{A}$. Hence,
$x\mathcal{A}=(aw)\mathcal{A}$. On the other hand,
$x^*=[x(wawx)]^*=wawxx^*=wawa(wx)^2\in (waw)\mathcal{A}$. Then $x^*\mathcal{A}\subseteq (waw)\mathcal{A}$.
Moreover, we have $$\begin{array}{rll}
waw&=&w(aw)=wxw(aw)^2=wa(wx)^2w(aw)^2\\
&=&(wawx)[wxw(aw)^2]=(wawx)^*[wxw(aw)^2]\\
&\in& x^*\mathcal{A}.
\end{array}$$ Therefore $x^*\mathcal{A}=(waw)\mathcal{A}$.

$(2)\Rightarrow (3)\Rightarrow (4)$ These are obvious.

$(4)\Rightarrow (1)$ Since $1-xwaw\in ^0(x)$, we see that $1-xwaw\in ^0(aw)$. Thus, $aw=xw(aw)^2$.
Moreover, we have $(xwawx)^*=x^*$, and so $1-wxwa\in ^0(x^*)$; hence, $1-wawx\in ^0(waw)$.
This implies that $(w-wawx)aw=0$, and then $w-wawx\in ^0(aw)=^0(x)$. It follows that
$wx=wa(wx)^2$. Therefore we deduce that
$$x=xwawx=(xwa)[wa(wx)^2]=x(wa)^2(wx)^2=[xw(aw)^2]xwx=a(wx)^2,$$ as required.
\end{proof}

\begin{cor} Let $a\in \mathcal{A}$ has $w$-weighted core inverse $x$. Then $a$ has generalized $w$-core-EP inverse and
$a^{\tiny\textcircled{d},w}=x$.\end{cor}
\begin{proof} This is obvious by Theorem 2.1.\end{proof}

\begin{cor} Let $a,x\in \mathcal{A}$. Then the following are equivalent:\end{cor}
\begin{enumerate}
\item [(1)] $a^{\tiny\textcircled{\#},w}=x$.
\item [(2)] $a(wx)^2=x, xw(aw)^2=aw, [(waw)x]^*=(waw)x, (waw)x(waw)=waw, x(waw)x=x$.
\end{enumerate}
\begin{proof} $(1)\Rightarrow (2)$ Since $a(wx)^2=x$, we see that $x=(aw)xwx=xw(aw)^2xwx=xwaw[a(wx)^2]=xwawx$. Moreover, we have
$$\begin{array}{rll}
(waw)x(waw)&=&wawxw(aw)=waw(xw)^2(aw)^2\\
&=&w[a(wx)^2]w(aw)^2=w(xw)(aw)^2=waw,
\end{array}$$ as required

$(2)\Rightarrow (1)$ is trivial.\end{proof}

If $a,x$ satisfy the equations $a=axa$ and $(ax)^*=ax$, then $x$ is called $(1,3)$-inverse of $a$ and is denoted by $a^{(1,3)}$. We use $\mathcal{A}^{(1,3)}$ to stand for sets of all $(1,3)$-invertible elements in $\mathcal{A}$.

\begin{thm} Let $a\in \mathcal{A}$. Then the following are equivalent:\end{thm}
\begin{enumerate}
\item [(1)] $a\in \mathcal{A}^{\tiny\textcircled{\#},w}$.
\vspace{-.5mm}
\item [(2)] $aw\in \mathcal{A}^{\#}$ and $waw\in \mathcal{A}^{(1,3)}$.
\end{enumerate}
\begin{proof} $(1)\Rightarrow (2)$ In view of Corollary 2.3,
there exists $x\in \mathcal{A}$ such that $a(wx)^2=x, xw(aw)^2=aw, (wawx)^*=wawx, (waw)x(waw)=waw, x(waw)x$ $=x$.
Hence, $waw\in \mathcal{A}^{(1,3)}$. In view of Corollary 2.2, $wa\in \mathcal{A}^d$. By using Cline's formula,
$aw\in \mathcal{A}^d$. Then $$\begin{array}{rll}
aw-(aw)^2(aw)^d&=&(xw)^{n-1}(aw)^{n}-(xw)^{n-1}(aw)^{n+1}(aw)^d\\
&=&(xw)^{n-1}[(aw)^{n}-(aw)^{n+1}(aw)^d].
\end{array}$$
Hence, $$||aw-(aw)^2(aw)^d||^{\frac{1}{n}}\leq ||(xw)^{n-1}||^{\frac{1}{n}}||(aw)^{n}-(aw)^{n+1}(aw)^d|^{\frac{1}{n}}.$$
Since $$\lim\limits_{n\to \infty}||(aw)^{n}-(aw)^{n+1}(aw)^d||^{\frac{1}{n}}=0,$$ we have
$$\lim\limits_{n\to \infty}||aw-(aw)^2(aw)^d||^{\frac{1}{n}}=0.$$
Then $aw=(aw)^2(aw)^d\in (aw)^2\mathcal{A}\bigcap \mathcal{A}(aw)^2$.
Therefore $aw\in \mathcal{A}^{\#}$.

$(2)\Rightarrow (1)$ Let $x=(aw)^{\#}aw(waw)^{(1,3)}$. Then we check that
$$\begin{array}{rll}
awx&=&aw(aw)^{\#}aw(waw)^{(1,3)}=aw(waw)^{(1,3)},\\
wawx&=&waw(waw)^{(1,3)},\\
(wawx)^*&=&[waw(waw)^{(1,3)}]^*=waw(waw)^{(1,3)}=wawx,\\
a(wx)^2&=&(awx)wx=aw(waw)^{(1,3)}wx=aw(waw)^{(1,3)}w(aw)^{\#}aw(waw)^{(1,3)}\\
&=&(aw)^{\#}a[waw(waw)^{(1,3)}waw][(aw)^{\#}]^2aw(waw)^{(1,3)}\\
&=&(aw)^{\#}(aw)^2[(aw)^{\#}]^2aw(waw)^{(1,3)}=(aw)^{\#}aw(waw)^{(1,3)}=x,\\
xw(aw)^2&=&(aw)^{\#}aw(waw)^{(1,3)}w(aw)^2=[(aw)^{\#}]^2a[waw(waw)^{(1,3)}waw]aw\\
&=&[(aw)^{\#}]^2(aw)^3=aw.
\end{array}$$ Therefore $a^{\tiny\textcircled{\#},w}=x$, as asserted.\end{proof}

\begin{cor} Let $a\in \mathcal{A}$. Then the following are equivalent:\end{cor}
\begin{enumerate}
\item [(1)] $a\in \mathcal{A}^{\tiny\textcircled{\#},w}$.
\vspace{-.5mm}
\item [(2)] $aw\in \mathcal{A}^{\#}$ and there exists $x\in \mathcal{A}$ such that $xw(aw)^2=aw, (wawx)^*=wawx.$
\vspace{-.5mm}
\item [(3)] $aw\in \mathcal{A}^{\#}$ and there exists $x\in \mathcal{A}$ such that $xw(aw)=aw(aw)^{\#},$ $(wawx)^*=wawx.$
\end{enumerate}
\begin{proof} $(1)\Rightarrow (2)$ This is obvious by Theorem 2.4 and Corollary 2.3.

$(2)\Rightarrow (3)$ Since $(aw)^2(aw)^{\#}=aw$, this implication is clear.

$(3)\Rightarrow (1)$ By assumption, $aw\in \mathcal{A}^{\#}$ and there exists $x\in \mathcal{A}$
such that $xw(aw)=aw(aw)^{\#}, (wawx)^*=wawx.$ Hence, we derive that $$(waw)x(waw)=(waw)[aw(aw)^{\#}]=w[(aw)^2(aw)^{\#}]=waw.$$
Therefore we complete the proof by Theorem 2.4.\end{proof}

\begin{lem} Let $a\in \mathcal{A}$. Then the following are equivalent:\end{lem}
\begin{enumerate}
\item [(1)] $a\in \mathcal{A}^{\tiny\textcircled{\#},w}$.
\vspace{-.5mm}
\item [(2)] $aw\in \mathcal{A}^{\#}$ and there exists $x\in \mathcal{A}$ such that $$x(waw)x=x, (waw)x(waw)=waw, im(x)=im(aw), im(x^*)\subseteq im(waw).$$
\end{enumerate}
\begin{proof} $(1)\Rightarrow (2)$ Set $x=a^{\tiny\textcircled{\#},w}.$ In view of Theorem 2.4, $aw\in \mathcal{A}^{\#}$ and
$(aw)^{\tiny\textcircled{\#}}=xw$.
In view of Corollary 2.3, we have $x=xwawx$ and $waw=(waw)x(waw)$. Thus $xw=a(wx)^2w=aw(xw)^2\in im(aw)$.
Hence, $im(xw)\subseteq im(aw)$. On the other hand, $aw=xw(aw)^2\in im(xw)$; hence, $im(aw)\subseteq im(xw)$.
Thus, $im(xw)=im(aw)$. Moreover, we have $x^*=(xwawx)^*=wawxx^*$, and then $imx^*\subseteq im(waw)$, as required.

$(2)\Rightarrow (1)$ By hypothesis, $aw\in \mathcal{A}^{\#}$ and there exists $x\in \mathcal{A}$ such that
$$x(waw)x=x, (waw)x(waw)=waw, im(xw)=im(aw), imx^*\subseteq im(waw).$$
Then $xwaw\mathcal{A}=xwaw\mathcal{A}=aw\mathcal{A}$. Since $(1-xwaw)xwaw=0$, we see that $(1-xwaw)aw=0$.
This implies that $aw=xw((aw)^2$. As $(waw)x(waw)=waw$, we have $(1-wawx)waw=0$; hence, $(1-wawx)x^*=0$. This implies that
$x^*=wawxx^*$. Thus $(wawx)^*=x^*(waw)^*=[wawxx^*](waw)^*=wawx(wawx)^*$. This implies that
$wawx=(wawx)^*(wawx)^*$. Accordingly, $(wawx)^*=wawx$, as desired.\end{proof}

Let $a,b,c\in \mathcal{A}$. An element $a$ has $(b,c)$-inverse provide that there exists $x\in \mathcal{A}$ such that $$xab=b, cax=c ~\mbox{and}~ x\in b\mathcal{A}x\bigcap x\mathcal{A}c.$$ If such $x$ exists, it is unique and denote it by $a^{(b,c)}$ (see~\cite{D}).

\begin{thm} Let $a\in \mathcal{A}$. Then the following are equivalent:\end{thm}
\begin{enumerate}
\item [(1)] $a\in \mathcal{A}^{\tiny\textcircled{\#},w}$.
\vspace{-.5mm}
\item [(2)] $aw\in \mathcal{A}^{\#}$ and $waw$ has $((aw)^{D}, ((wa)^{D})^*)$-inverse.
\end{enumerate}
In this case, $a^{\tiny\textcircled{\#},w}=(waw)^{((aw)^{D},((wa)^D)^*)}.$
\begin{proof} $(1)\Rightarrow (2)$ In view of Lemma 2.6, $aw\in \mathcal{A}^{\#}$. By using Cline's formula, $wa\in \mathcal{A}^D$.
Let $x=a^{\tiny\textcircled{\#},w}$.
Then we verify that $$\begin{array}{rll}
x(waw)(aw)^{\#}&=&xw(aw)(aw)^{\#}=xw(aw)^2[(aw)^{\#}]^2\\
&=&aw[(aw)^{\#}]^2=(aw)^{\#},\\
((wa)^D)^*(waw)x&=&((wa)^D)^*(wawx)^*=[wawx(wa)^D]^*\\
&=&[w(awxwa)((wa)^D)^2]^*=[wa((wa)^D)^2]^*=((wa)^D)^*,\\
x&=&xwawx.
\end{array}$$
Obviously, $x=a(wx)^2=(aw)xwx=(aw)^{\#}(aw)^2xwx$.
Hence $x\in (aw)^{\#}\mathcal{A}$. Moreover, we have
$$\begin{array}{rll}
x&=&x(wawx)=x(wawx)^*=x(wx)^*(wa)^*\\
&=&x(wx)^*((wa)^2)^*((wa)^D)^*\in \mathcal{A}((wa)^D)^*.
\end{array}$$ Therefore
$x\in (aw)^{\#}\mathcal{A}x\bigcap x\mathcal{A}((wa)^D)^*$. Hence,
$waw$ has $((aw)^{\#}, ((wa)^{D})^*)$-inverse $x$, and so $a^{\tiny\textcircled{\#},w}=(waw)^{((aw)^{D},((wa)^D)^*)}.$

$(2)\Rightarrow (1)$ Let $x=(waw)^{((aw)^D, ((wa)^D)^*)}$. Then we can find some $s,t\in \mathcal{A}$ such that
$$x=(aw)^Dsx=xt((wa)^D)^*,xwaw(aw)^D=(aw)^D, ((wa)^D)^*wawx=((wa)^D)^*.$$

Claim 1. $xwawx=x$. We verify that $$\begin{array}{rll}
xwawx&=&[xt((wa)^D)^*]wawx\\
&=&xt[((wa)^D)^*wawx]\\
&=&xt((wa)^D)^*\\
&=&x.
\end{array}$$

Claim 2. $(waw)x(waw)=waw$. As $xwaw(aw)^D=(aw)^D$, we see that $xwaw(aw)^D(aw)=(aw)^D(aw)$. Hence,
$(waw)x(waw)=(waw)(aw)^D(aw)=waw$.

Claim 3. $im(x)\subseteq im(aw)$. Clearly, $xw=(aw)^Dsxw=aw[((aw)^D)^2sxw]$, and so  $x=xwawx\in im(aw)$.
Hence $im(x)\subseteq im(aw)$. On the other hand, $aw=(aw)^D(aw)^2=[xwaw(aw)^D](aw)^2\in im(xw)$; hence, $im(aw)\subseteq im(xw)=im(x)$, as desired.

Claim 4. $im(x^*)\subseteq im(waw)$. Since $x=xt((wa)^D)^*=xt[wawa((wa)^D)^3]^*=xt[a((wa)^D)^3]^*(waw)^*$, we have $x^*=(waw)[a((wa)^D)^3](xt)^*$.
This implies that $imx^*\subseteq im(waw)$.

Therefore we conclude that $a\in \mathcal{A}^{\tiny\textcircled{\#},w}$ by Lemma 2.6.\end{proof}

As an immediate consequence, we now give an improvement of~\cite[Theorem 2.3]{MD4}.

\begin{cor} Let $a\in \mathcal{A}$. Then the following are equivalent:\end{cor}
\begin{enumerate}
\item [(1)] $a\in \mathcal{A}^{\tiny\textcircled{\#}}$.
\vspace{-.5mm}
\item [(2)] $a\in \mathcal{A}^{\#}$ and $a$ has $(a^{\#}, (a^{\#})^*)$-inverse.
\end{enumerate}
In this case, $a^{\tiny\textcircled{\#}}=a^{(a^{\#}, (a^{\#})^*)}.$
\begin{proof} This is obvious by choosing $w=1$ in Theorem 2.7.\end{proof}

\section{weighted generalized core-EP decomposition}

The purpose of this section is to characterize $w$-weighted generalized core-EP inverse in a Banach *-algebra by using weight core inverse and quasinilpotent.
The following theorem contains new characterizations for a $w$-weighted generalized core-EP inverse.

\begin{thm} Let $a,w\in \mathcal{A}$. Then the following are equivalent:\end{thm}
\begin{enumerate}
\item [(1)] $a\in \mathcal{A}^{\tiny\textcircled{d},w}$.
\item [(2)] There exist $z,y\in \mathcal{A}$ such that $$a=z+y, (wz)^*(wy)=ywz=0, z\in \mathcal{A}^{\tiny\textcircled{\#},w}, y\in \mathcal{A}_w^{qnil}.$$
\item [(3)] There exist $z,y\in \mathcal{A}$ such that $$a=z+y, ywz=0, z\in \mathcal{A}^{\tiny\textcircled{\#},w}, y\in \mathcal{A}_w^{qnil}.$$
\end{enumerate}
In this case, $a^{\tiny\textcircled{d},w}=x=z^{\tiny\textcircled{\#},w}.$
\begin{proof} $(1)\Rightarrow (2)$ By hypotheses, there exists $x\in \mathcal{A}$ such that $$\begin{array}{c}
a(wx)^2=x, (wawx)^*=wawx,\\
\lim\limits_{n\to \infty}||(aw)^n-(xw)(aw)^{n+1}||^{\frac{1}{n}}=0.
\end{array}$$
Set $z=awxwa$ and $y=a-awxwa.$ Then $a=z+y$.
Since $$\lim\limits_{n\to \infty}||(aw)^n-(xw)(aw)^{n+1}||^{\frac{1}{n}}=0,$$ we prove that
$$xwawx=x, xw(aw)^2x=awx.$$

We claim that $z$ has $w$-weighted core inverse $x$. Evidently, we verify that
$$\begin{array}{rll}
z(wx)^2&=&(awxwa)(wx)^2=awxw[a(wx)^2]=(awxw)x=a(wx)^2=x,\\
wzwx&=&w(awxwa)wx=(wawx)^2,\\
(wzwx)^*&=&[(wawx)^*]^2=(wawz)^2=wzwx,\\
xw(zw)^2&=&xwawxwawawxwaw=xwaw[xw(aw)^2]xwaw\\
&=&xw(aw)^2xwaw=awxwaw=zw.
\end{array}$$ Thus, $z\in \mathcal{A}^{\tiny\textcircled{\#},w}$ and $x=z^{\tiny\textcircled{\#},w}$.

Moreover, we see that $$(aw-xw(aw)^2)x=0.$$
Then we verify that
$$\begin{array}{rl}
&||(aw-xw(aw)^2)^{n+2}||^{\frac{1}{n+1}}\\
\leq&||1-xw(aw)||^{\frac{1}{n+1}}||aw(aw-xw(aw)^2)^{n+1}||^{\frac{1}{n+1}}\\
=&||1-xw(aw)||^{\frac{1}{n+1}}||aw(aw-xw(aw)^2)^n(aw-xw(aw)^2)||^{\frac{1}{n+1}}\\
=&||1-xw(aw)||^{\frac{1}{n+1}}||aw(aw-xw(aw)^2)^{n-1}(aw-xw(aw)^2)aw||^{\frac{1}{n+1}}\\
=&||1-xw(aw)||^{\frac{1}{n+1}}||aw(aw-xw(aw)^2)^{n-1}(aw)^2||^{\frac{1}{n+1}}\\
\vdots&\\
=&||1-xw(aw)||^{\frac{1}{n+1}}||aw(aw-xw(aw)^2)(aw)^n||^{\frac{1}{n+1}}\\
\leq &||1-xw(aw)||^{\frac{1}{n+1}}\big[||(aw)^{n+1}-(aw)(xw)(aw)^{n+1}||^{\frac{1}{n}}\big]^{\frac{n}{n+1}}||aw||^{\frac{1}{n+1}}.
\end{array}$$ Accordingly, $$\lim\limits_{n\to \infty}||(aw-xw(aw)^2)^{n+2}||^{\frac{1}{n+2}}=0.$$ This implies that $aw-xw(aw)^2\in \mathcal{A}^{qnil}$.
By using Cline's formula (see~\cite[Theorem 2.1]{L}), $yw=aw-awxwaw\in \mathcal{A}^{qnil}$. Hence, $y\in \mathcal{A}_w^{qnil}.$

Moreover, we verify that $$\begin{array}{rll}
(wz)^*wy&=&(wawxwa)^*(wa-wawxwa)=(wa)^*(wawx)^*(1-wawx)wa\\
&=&(wa)^*(wawx)(1-wawx)wa=(wa)^*(wawxw-wawxwawxw)a\\
&=&(wa)^*[wawxw-w(awxwawxw)]a\\
&=&(wa)^*[wawxw-w(awxw)]a=0,\\
ywx&=&(a-awxwa)w(awxwa)=awawxwa-aw[xw(aw)^2]xwa\\
&=&awawxwa-(aw)^2xwa=0,
\end{array}$$ as desired.

$(2)\Rightarrow (3)$ This is trivial.

$(3)\Rightarrow (1)$ By hypothesis, there exist $z,y\in \mathcal{A}$ such that $$a=z+y, ywz=0, z\in
\mathcal{A}^{\tiny\textcircled{\#},w}, y\in \mathcal{A}_w^{qnil}.$$ Set $x=z^{\tiny\textcircled{\#},w}$. Then $$\begin{array}{rll}
z(wx)^2&=&x,\\
(wzwx)^*&=&wzwx,\\
xw(zw)^2&=&zw.
\end{array}$$

It is easy to verify that
$$\begin{array}{rll}
a(wx)^2&=&(z+y)(wx)^2=z(wx)^2=x,\\
(waw)x&=&w(z+y)wx=wzwx,\\
((waw)x)^*&=&(wzwx)^*=wzwx=(waw)x.
\end{array}$$
Thus, $wa(wx)^2=wx, (wawx)^*=wawx$. We verify that
$$\begin{array}{rl}
&(aw)^2-(aw)(xw)(aw)^2\\
=&aw(zw+yw)-(aw)(xwzw+xwyw)(zw+yw)\\
=&aw(zw+yw)-(aw)[(xwzw)zw+(xwzw)yw+xw(ywz)w+xw(yw)^2]\\
=&aw(zw+yw)-[awxw(zw)^2+(z+y)wxwzwyw+awxw(yw)^2]\\
=&awzw+awyw-[awzw+zwxw(zw)yw+awxw(yw)^2]\\
=&awzw+awyw-[awzw+zwxw(xw(zw)^2)yw+awxw(yw)^2]\\
=&awyw-z(wx)^2w(zw)^2yw-awxw(yw)^2]\\
=&awyw-[xw(zw)^2]yw-awxw(yw)^2]\\
=&(z+y)wyw-zwyw-(z+y)wxw(yw)^2]\\
=&(yw)^2-zwxw(yw)^2\\
=&(1-zwxw)(yw)^2.
\end{array}$$
Since $yw(aw)=yw(zw+yw)=(yw)^2$, we deduce that
$$\begin{array}{rll}
(aw)^n-(aw)(xw)(aw)^n&=&[(aw)^2-(aw)(xw)(aw)^2](aw)^{n-2}\\
&=&[(1-zwxw)(yw)^2](aw)^{n-2}\\
&=&(1-zwxw)(yw)[yw(aw)^{n-2}]\\
&=&(1-zwxw)(yw)(yw)^{n-1}.
\end{array}$$
Then $$\begin{array}{rll}
||(aw)^n-(aw)xw(aw)^{n}||^{\frac{1}{n}}
&\leq&||1-zwxw||||(yw)^{n}||.
\end{array}$$ Since $y\in \mathcal{A}_w^{qnil}$, we see that $$\lim\limits_{n\to \infty}||(yw)^{n}||^{\frac{1}{n}}=0.$$
Therefore $$\lim\limits_{n\to \infty}||(aw)^n-(aw)(xw)(aw)^n||^{\frac{1}{n}}=0.$$

Since $$\begin{array}{rl}
&||(wa)^{n+1}-(wa)(wx)(wa)^{n+1}||^{\frac{1}{n+1}}\\
=&||w[(aw)^{n}-(aw)(xw)(aw)^{n}]a||^{\frac{1}{n+1}}\\
\leq&||w||^{\frac{1}{n+1}}\big(||(aw)^{n}-(aw)(xw)(aw)^{n}||^{\frac{1}{n}}\big)^{\frac{n}{n+1}}||a||^{\frac{1}{n+1}},
\end{array}$$ we deduce that $$\lim\limits_{n\to \infty}||(wa)^{n+1}-(wa)(wx)(wa)^{n+1}||^{\frac{1}{n+1}}.$$
In view of Theorem 2.4,
$zw\in \mathcal{A}^{\#}$. By virtue of Cline's formula, $wz\in \mathcal{A}^d$. Clearly, $wy\in \mathcal{A}^{qnil}$. Since $ywz=0$, it follows by~\cite[Lemma 15.2.2]{CM2} that $wa=(wz)+(wy)\in \mathcal{A}^d$. According to~\cite[Theorem 2.5]{CM3}, $wa\in \mathcal{A}^{\tiny\textcircled{d}}$. Therefore $a\in \mathcal{A}^{\tiny\textcircled{d},w}$ by ~\cite[Theorem 2.1]{C}.\end{proof}

\begin{cor} Let $a\in \mathcal{A}^{\tiny\textcircled{d},w}$ and $b\in \mathcal{A}^{qnil}$. If $bwa=0$, then
$a+b\in \mathcal{A}^{\tiny\textcircled{d},w}$. In this case, $(a+b)^{\tiny\textcircled{d},w}=a^{\tiny\textcircled{d},w}.$\end{cor}
\begin{proof} Since $a\in \mathcal{A}^{\tiny\textcircled{d},w}$, by virtue of Theorem 3.1, there exist $x\in \mathcal{A}^{\tiny\textcircled{\#},w}$ and $y\in \mathcal{A}_w^{qnil}$ such that $a=x+y, ywx=0$. As in the proof of Theorem 3.1, $x=awa^{\tiny\textcircled{d},w}wa$ and $y=a-awa^{\tiny\textcircled{d},w}wa$.
Then $a=x+(y+b)$. As $bwy=bw(a-awa^{\tiny\textcircled{d},w}wa)=0$, it follows by ~\cite[Lemma 2.4]{CK} that $(y+b)w\in \mathcal{A}^{qnil}$. Hence, $y+b\in \mathcal{A}_w^{qnil}$.
Obviously, $(y+b)wx=ywx+bwx=0$.
In light of Theorem 3.1, $a+b\in \mathcal{A}^{\tiny\textcircled{d},w}$. In this case, $(a+b)^{\tiny\textcircled{d},w}=x^{\tiny\textcircled{\#},w}=a^{\tiny\textcircled{d},w}.$\end{proof}

\begin{cor} Let $a\in \mathcal{A}$. Then the following are equivalent:\end{cor}
\begin{enumerate}
\item [(1)] $a\in \mathcal{A}^{\tiny\textcircled{d}}$.
\item [(2)] There exist $z,y\in \mathcal{A}$ such that $$a=z+y, z^*y=yz=0, z\in \mathcal{A}^{\tiny\textcircled{\#}}, y\in \mathcal{A}^{qnil}.$$
\item [(3)] There exist $z,y\in \mathcal{A}$ such that $$a=z+y, yz=0, z\in \mathcal{A}^{\tiny\textcircled{\#}}, y\in \mathcal{A}^{qnil}.$$
\end{enumerate}
In this case, $a^{\tiny\textcircled{d},w}=x=z^{\tiny\textcircled{\#}}.$
\begin{proof} This is obvious by choosing $w=1$ in Theorem 3.1.\end{proof}

For a Hilbert space operator $T$, it follows by Corollary 3.3 that $T$ has generalized Drazin inverse if and only if it has core-EP inverse (see~\cite{M2,MD}).

\begin{cor} Let $a\in \mathcal{A}$. Then the following are equivalent:\end{cor}
\begin{enumerate}
\item [(1)] $a\in \mathcal{A}^{\tiny\textcircled{D}}$.
\item [(2)] There exist $z,y\in \mathcal{A}$ such that $$a=z+y, z^*y=yz=0, z\in \mathcal{A}^{\tiny\textcircled{\#}}, y\in \mathcal{A}^{nil}.$$
\item [(3)] There exist $z,y\in \mathcal{A}$ such that $$a=z+y, yz=0, z\in \mathcal{A}^{\tiny\textcircled{\#}}, y\in \mathcal{A}^{nil}.$$
\end{enumerate}
In this case, $a^{\tiny\textcircled{d}}=x=z^{\tiny\textcircled{\#}}.$
\begin{proof} $(1)\Rightarrow (2)$ This is proved in ~\cite[Theorem 2.9]{G2}.

$(2)\Rightarrow (3)$ This is trivial.

$(3)\Rightarrow (1)$ In view of Corollary 3.3, $a\in \mathcal{A}^{\tiny\textcircled{d}}$. Since $y,z\in \mathcal{A}^D$ and $yz=0$, we easily see that
$a\in \mathcal{A}^D$. Therefore $a\in \mathcal{A}^{\tiny\textcircled{D}}$ by ~\cite[Corollary 3.4]{CM3}.\end{proof}

We present an example to illustrate Theorem 3.1.

\begin{exam}\end{exam} The space ${\ell}^2({\Bbb N})$ is a Hilbert space consisting of all square-summable infinite sequences of complex numbers, i.e.,
${\ell}^2({\Bbb N})=\{ (x_i)_{i\in {\Bbb N}}~|~\sum\limits_{i=1}^{\infty}|x_i|^2<\infty, x_1, x_2, \cdots \in {\Bbb C}\}.$ Let
$H={\ell}^2({\Bbb N})\bigoplus {\ell}^2({\Bbb N})$ be the Hilbert space of the direct sum of ${\ell}^2({\Bbb N})$ and itself.
Let $\sigma$ be defined on ${\ell}^2({\Bbb N})$ by:
$$\sigma(x_1,x_2,x_3,\cdots )=(1^{\frac{1}{1}}x_1+1^{\frac{1}{1}}x_2,2^{\frac{1}{2}}x_2,0,\cdots );$$ $\tau$ is defined
on ${\ell}^2({\Bbb N})$ by:
$$u(x_1,x_2,x_3,\cdots )=(x_1/1^{\frac{1}{1}}, x_2/2^{\frac{1}{2}}, x_3/3^{\frac{1}{3}}, x_4/4^{\frac{1}{4}},\cdots )$$ and
$\tau $ is defined on ${\ell}^2({\Bbb N})$ by:
$$\tau(x_1,x_2,x_3,\cdots )=(1^{\frac{1}{1}}x_1-1^{\frac{1}{1}}x_2,2^{\frac{1}{2}}x_2,0,0,\cdots ).$$ Clearly, $\sigma$ and $\tau$ are well defined.

It is easy to verify that
$$\lim\limits_{n\to \infty}||(x_n)^2||/||( x_n/\sqrt[n]{n})^2||=\lim\limits_{n\to \infty}(\sqrt[n]{n})^2=1.$$ Hence,
$\sum\limits_{n=1}^{\infty}\big(x_n/\sqrt[n]{n}\big)^2$ converges. Then $u$ is well defined.

Let $\gamma $ be defined on ${\ell}^2({\Bbb N})$ by:
$$\gamma (x_1,x_2,x_3,\cdots )=(x_2,x_3,x_4,x_4,\cdots );$$ $v$ is defined
on ${\ell}^2({\Bbb N})$ by:
$$v(x_1,x_2,x_3,\cdots )=(\frac{1}{3}x_1,\frac{1}{4}x_2,\frac{1}{5}x_3,\cdots ).$$ Let
$\delta=v\gamma$. Then $\delta$ be defined on ${\ell}^2({\Bbb N})$ by: $$\delta (x_1,x_2,x_3,\cdots )=(\frac{1}{3}x_2,\frac{1}{4}x_3,\frac{1}{5}x_4,\cdots ).$$
Let $e_n=(\underbrace{0,0,\cdots ,0,1,}_{n} 0,0,\cdots )$. Then $\delta(e_n)=\frac{1}{n+2}e_{n+1}$. One directly checks that $$\delta^k(e_n)=\prod\limits_{j=0}^{k-1}\frac{1}{n+j+2}e_{n+k}
=\frac{(n+1)!}{(n+k+1)!}e_{n+k}.$$ Then we derive that $$||\delta^k||=\frac{1}{(k+1)!}.$$ By virtue of Stirling's formula, we have $(k+1)!\sim \sqrt{2\pi (k+1)}\big(\frac{k+1}{e}\big)^{k+1}$, and then
$$\lim\limits_{k\to \infty}||\delta^k||^{\frac{1}{k}}=\lim\limits_{k\to \infty}\frac{e^{1+\frac{1}{k}}}{(2\pi)^{\frac{1}{k}}[(k+1)^{\frac{1}{k+1}}]^{1+\frac{1}{k}}(k+1)^{1+\frac{1}{k}}}=0.$$ Thus $\delta$ is quasinilpotent.

Let $T=\sigma \oplus \gamma $ be the direct sum of $\sigma $ and $\gamma $, $W=u\oplus v$ be the direct sum of $u$ and $v$,
$S=\tau\oplus 0$ be the direct sum of $\tau$ and $0$. Then $T,W$ and $S$ are operators on $H$.
We see that $T=\alpha+\beta$, where $\alpha=\sigma \oplus 0$ and $\beta=0\oplus \gamma$.

Claim 1. $\alpha $ has $W$-core inverse. We directly check that $u\sigma u\tau $ is given by
$$u\sigma u\tau(x_1,x_2,x_3,\cdots )=(x_1,x_2,0,0,\cdots ).$$
Hence, $W\alpha WS=I\oplus 0$. Moreover, we verify that
$$\tau u\sigma u\sigma=\sigma, \sigma u\tau u\tau =\tau.$$
This implies that
$$\begin{array}{rll}
\alpha (WS)^2&=&S,\\
(W\alpha WS)^*&=&W\alpha WS,\\
SW(\alpha W)^2&=&\alpha W.
\end{array}$$ Then $S=\alpha^{\tiny\textcircled{\#},w}$.

Claim 2. $W\beta$ is quasinilpotent. Since $W\beta=0\oplus \delta$, we see that $W\beta$ is quasinilpotent.

Claim 3. $\beta W\alpha=0$. This is obvious.

Therefore $T$ has a generalized $W$-core inverse by Theorem 3.1.\\

As it is well known, an element $a\in \mathcal{A}$ has g-Drazin inverse if and only if it is quasi-polar, i.e., there exists an idempotent $p\in \mathcal{A}$ such that $pa=ap\in \mathcal{A}^{qnil}, a+p\in \mathcal{A}^{-1}$ (see~\cite[Corollary 15.1.5]{CM2}). For $w$-weighted generalized core-EP inverse, we establish the following polar-like characterization.

\begin{thm} Let $a,w\in \mathcal{A}$ and $n\in {\Bbb N}$. Then $a\in \mathcal{A}^{\tiny\textcircled{d},w}$ if and only if\end{thm}
\begin{enumerate}
\item [(1)]{\it $a\in \mathcal{A}^{d,w}$;}
\item [(2)]{\it There exists a projection $p\in \mathcal{A}$ such that $$pwa=pwap\in \mathcal{A}^{qnil}, ~\mbox{and} ~(wa)^n+p\in \mathcal{A}^{-1}.$$}
\end{enumerate}
\begin{proof} $\Longrightarrow $ Since $a\in \mathcal{A}^{\tiny\textcircled{d},w}$, by using Theorem 3.1, there exist $x,y\in \mathcal{A}$ such that $$a=x+y, ywx=0, x\in
\mathcal{A}^{\tiny\textcircled{\#},w}, y\in \mathcal{A}_w^{qnil}.$$
Let $z=x^{\tiny\textcircled{\#},w}$. By virtue of Corollary 2.3, we have
$$\begin{array}{c}
x(wz)^2=z, [(wxw)z]^*=(wxw)z, (wxw)z(wxw)=wxw,\\
zw(xw)^2=xw, z(wxw)z=z.
\end{array}$$ Let $p=1-wxwz$. Then $p^2=p=p^*$ and $zp=0$.
We directly check that
$$\begin{array}{rl}
&[(wx)^n+1-wxwz][(wz)^n+1-wxwz]\\
=&(wx)^n(wz)^n+(wx)^n(1-wxwz)+(1-wxwz)(wz)^n+(1-wxwz)^2\\
=&(wx)(wz)+(wx)^n(1-wxwz)+[wz-wx(wz)^2](wz)^{n-1}+(1-wxwz)\\
=&1+(wx)^n(1-wxwz).
\end{array}$$

Case 1. $n=1$.  $$\begin{array}{rl}
&[wz+1-wzwx][wx+1-wxwz]\\
=&wzwx+wz(1-wxwz)+(1-wzwx)wx+(1-wzwx)(1-wxwz)\\
=&1+(1-wzwx)wx-wxwz+wz(wx)^2wz\\
=&1+(1-wzwx)wx.
\end{array}$$

Case 2. $n\geq 2$. Since $(wx)^n, 1-wxwz\in \mathcal{A}^d$ and $(1-wxwz)(wx)^n=0$, it follows by~\cite[Lemma 15.2.2]{CM2} that
$s=(wx)^n+(1-wxwz)\in \mathcal{A}^d$. $s[(wz)^n+1-wxwz][1+(wx)^n(1-wxwz)]^{-1}=1$. Since
$\lim\limits_{n\to \infty}||s^m-s^ds^{m+1}||^{\frac{1}{n}}=0,$ we deduce that $s^ds=1$. Hence, $s$ is right and left invertible,
That is, $(wx)^n+p$ is right and left invertible. Thus $$(wx)^n+p\in \mathcal{A}^{-1}.$$
Since $$\begin{array}{rl}
&[\sum\limits_{i=1}^{n}(wx)^{n-i}(wy)^i][(wx)^n+p]\\
=&[\sum\limits_{i=1}^{n}(wx)^{n-i}(wy)^i][(wx)^n+1-wxwz]\\
=& \sum\limits_{i=1}^{n}(wx)^{n-i}(wy)^i\in \mathcal{A}^{-1},
\end{array}$$ we see that $[\sum\limits_{i=1}^{n}(wx)^{n-i}(wy)^i][(wx)^n+p]^{-1}=\sum\limits_{i=1}^{n}(wx)^{n-i}(wy)^i\in \mathcal{A}^{qnil}$.
By virtue of Cline's formula, $[(wx)^n+p]^{-1}[\sum\limits_{i=1}^{n}(wx)^{n-i}(wy)^i]\in \mathcal{A}^{qnil}$. Hence, $1+[(wx)^n+p]^{-1}[\sum\limits_{i=1}^{n}(wx)^{n-i}(wy)^i]\in \mathcal{A}^{-1}$.
Therefore, we check that
$$\begin{array}{rll}
pwa&=&pw(x+y)=(1-wxwz)wx+pwy\in \mathcal{A}^{qnil},\\
pwa(1-p)&=&pw(x+y)wxwz=pwxwxwz=(1-wxwz)wxwxwz=0,\\
pwa&=&pwap,\\
(wa)^n+p&=&(wx+wy)^n+p=(wx)^n+\sum\limits_{i=1}^{n}(wx)^{n-i}(wy)^i+p\\
&=&[(wx)^n+p]+\sum\limits_{i=1}^{n}(wx)^{n-i}(wy)^i\\
&=&[(wx)^n+p][1+((wx)^n+p)^{-1}\sum\limits_{i=1}^{n}(wx)^{n-i}(wy)^i]\in\mathcal{A}^{-1},
\end{array}$$ as required.

$\Longleftarrow$ Obviously, $wa\in \mathcal{A}^d$; hence, $(wa)^n\in \mathcal{A}^d$. By hypothesis, there exists a projection $p\in \mathcal{A}$ such that $$1-p\in w\mathcal{A}, pwa=pwap\in \mathcal{A}^{qnil}, ~\mbox{and} ~u:=(wa)^n+p\in \mathcal{A}^{-1}.$$ Since $pwa=pwap$, we have $$p(wa)^2p=(pwa)(wap)=pwa(pwap)=(pwap)wa=p(wa)^2.$$ By induction,
we have $p(wa)^n=p(wa)^np$. In view of ~\cite[Theorem 2.1]{CM4}, $(wa)^n\in \mathcal{A}^{\tiny\textcircled{d}}$. We easily check that
$(wa)^{\tiny\textcircled{d}}=(wa)^{n-1}[(wa)^n]^{\tiny\textcircled{d}}$. Therefore $a\in \mathcal{A}^{\tiny\textcircled{d},w}$ by~\cite[Theorem 2.1]{C}.\end{proof}

\begin{cor} Let $a\in \mathcal{A}$ and $n\in {\Bbb N}$. Then $a\in \mathcal{A}^{\tiny\textcircled{d}}$ if and only if\end{cor}
\begin{enumerate}
\item [(1)]{\it $a\in \mathcal{A}^d$;}
\item [(2)]{\it There exists a projection $p\in \mathcal{A}$ such that $$pa=pap\in \mathcal{A}^{qnil}, ~\mbox{and} ~a^n+p\in \mathcal{A}^{-1}.$$}
\end{enumerate}
\begin{proof} Straightforward by choosing $w=1$ in Theorem 3.6.\end{proof}

\begin{cor} Let $A,W\in {\Bbb C}^{n\times n}$ and $m\in {\Bbb N}$. Then $(WA)^m+I_n-WAWA^{\tiny\textcircled{\dag},W}$ is invertible.
\end{cor}
\begin{proof} This is obvious by Theorem 3.6.\end{proof}

\section{representations of $w$-weighted generalized core-EP inverse}

This section aims to characterize the generalized weighted core inverse using other weighted generalized inverses, and its representations are presented accordingly. The following lemma is crucial.

\begin{lem} Let $a,w\in \mathcal{A}$. Then the following are equivalent:\end{lem}
\begin{enumerate}
\item [(1)] $a\in \mathcal{A}^{\tiny\textcircled{d},w}$.
\item [(2)] $a\in \mathcal{A}^{d,w}$ and there exist $x\in \mathcal{A}$ such that
$$a(wx)^2=x, (wawx)^*=wawx,\lim\limits_{n\to \infty}||(aw)^n-(ax)(xw)(aw)^n||^{\frac{1}{n}}=0.$$
\end{enumerate}
In this case, $a^{\tiny\textcircled{d},w}=x.$
\begin{proof}  $(1)\Rightarrow (2)$ By Theorem 3.6, $a\in \mathcal{A}^{\tiny\textcircled{d},w}$. Set $x=a^{\tiny\textcircled{d},w}$. Then we have
$$a(wx)^2=x, (wawx)^*=wawx, \lim\limits_{n\to \infty}||(aw)^{n-1}-(xw)(aw)^{n}||^{\frac{1}{n-1}}=0.$$
Obviously, $$||(aw)^n-(aw)(xw)(aw)^n||^{\frac{1}{n}}\leq ||aw||^{\frac{1}{n}}\big(||[(aw)^{n-1}-(xw)(aw)^{n}]||^{\frac{1}{n-1}}\big)^{1-\frac{1}{n}}.$$
Thus, $\lim\limits_{n\to \infty}||(aw)^{n}-(ax)(xw)(aw)^{n}||^{\frac{1}{n}}=0,$ as required.

$(2)\Rightarrow (1)$  We observe that
$$||(wa)^{n+1}-(wa)(wx)(wa)^{n+1}||\leq ||w||||(aw)^n-(aw)(xw)(aw)^n||||a|||.$$
Then $$\lim\limits_{n\to \infty}||(wa)^{n+1}-(wa)(wx)(wa)^{n+1}||^{\frac{1}{n+1}}=0.$$
Since $wa\in \mathcal{A}^d$, it follows by ~\cite[Theorem 2.5]{CM3} that $wa\in \mathcal{A}^{\tiny\textcircled{d}}$. In light of
~\cite[Theorem 2.1]{C}, $a\in \mathcal{A}^{\tiny\textcircled{d},w}$, as asserted.\end{proof}

\begin{thm} Let $a,w\in \mathcal{A}$. Then the following are equivalent:\end{thm}
\begin{enumerate}
\item [(1)] $a\in \mathcal{A}^{\tiny\textcircled{d},w}$.
\vspace{-.5mm}
\item [(2)] $a\in \mathcal{A}^{d,w}$ and $a^{d,w}\in\mathcal{A}^{\tiny\textcircled{\#},w}$.
\end{enumerate}
In this case, $$a^{\tiny\textcircled{d},w}=[a^{d,w}w]^2(a^{d,w})^{\tiny\textcircled{\#},w}.$$
\begin{proof} $(1)\Rightarrow (2)$ By virtue of Lemma 4.1, $a\in \mathcal{A}^{d,w}$.
Let $x=a^{\tiny\textcircled{d},w}$. Then we have
 $$x=a(wx)^2, (wawx)^*=wawx, \lim\limits_{n\to \infty}||(aw)^n-awxw(aw)^n||^{\frac{1}{n}}=0.$$
Obviously, we have
 $$\begin{array}{rl}
&||aw(aw)^d-awxw(aw)(aw)^d||\\
=&||(aw)^n[(aw)^d]^n-awxw(aw)^n[(aw)^d]^n||\\
\leq&||(aw)^n-awxw(aw)^n||||(aw)^d]^n||.
\end{array}$$ As $\lim\limits_{n\to \infty}||(aw)^n-awxw(aw)^n||^{\frac{1}{n}}=0$, we see that $$\lim\limits_{n\to \infty}||aw(aw)^d-awxw(aw)(aw)^d||^{\frac{1}{n}}=0.$$
Hence, $awxw(aw)(aw)^d=aw(aw)^d$. Let $z=(aw)^2x$. Then
$$\begin{array}{rll}
a^{d,w}wz&=&a^{d,w}w(aw)^2x=[(aw)^d]^2aw(aw)^2x=awx,\\
a^{d,w}(wz)^2&=&(awx)(wz)=(awxw)(aw)^2x\\
&=&[awxw(aw)^2]x=(aw)^2x=z,\\
(wa^{d,w}wz)^*&=&(wawx)^*=w(awx)=wa^{d,w}wz,\\
zw(a^{d,w}w)^2&=&(aw)[(aw)xw(aw)][(aw)^d]^3=(aw)^2[(aw)^d]^3=a^{d,w}w.
\end{array}$$
Accordingly, $$a^{d,w}(wz)^2=z, (wa^{d,w}wz)^*=wa^{d,w}wz, zw(a^{d,w}w)^2=a^{d,w}w.$$ Then $a^{d,w}\in \mathcal{A}^{\tiny\textcircled{\#},w}$ and $(a^{d,w})^{\tiny\textcircled{\#},w}=z=(aw)^2a^{\tiny\textcircled{d},w},$ as required.

$(2)\Rightarrow (1)$ Set $z=(a^{d,w})^{\tiny\textcircled{\#},w}$. Then we have
$$zw(a^{d,w}w)^2=a^{d,w}w, [wa^{d,w}wz]^*=wa^{d,w}wz, a^{d,w}(wz)^2=z.$$
Let $x=[a^{d,w}w]^2z$. Then we verify that

$$\begin{array}{rll}
awx&=&aw[a^{d,w}w]^2z=aw[a^{d,w}wa^{d,w}w]z=(aw)^dz,\\
a(wx)^2&=&(awx)wx=(aw)^dzw[a^{d,w}w]^2z=(aw)^d[zw(a^{d,w}w)^2z]\\
&=&(aw)^da^{d,w}wz=[(aw)^d]^2z=x, \\
wawx&=&w(aw)^dz=w([(aw)^d]^2a)wz=wa^{d,w}wz,\\
(wawx)^*&=&wawx, \\
\end{array}$$  Moreover, we see that
Since $zw(a^{d,w}w)^2=a^{d,w}w$, we have $zw[(aw)^d]^2=(aw)^d$
$$\begin{array}{rl}
&||(aw)^n-(aw)(xw)(aw)^n||\\
\leq &||(aw)^n-(aw)^d(zw)(aw)^d(aw)^{n+1}||+||(aw)^d(zw)(1-(aw)^d(aw))(aw)^n||\\
=&||(aw)^n-(aw)^d[(zw)((aw)^d)^2](aw)^{n+2}||+||(aw)^d(zw)[(aw)^n-(aw)^d(aw)^{n+1}]||\\
\leq&||(aw)^n-((aw)^d)^2(aw)^{n+2}||+||(aw)^d(zw)[(aw)^n-(aw)^d(aw)^{n+1}]||\\
\leq&[1+||(aw)^d(zw)||]||[aw-(aw)^d(aw)^2]^{n}||.
\end{array}$$

Hence, $$\lim\limits_{n\to \infty}||(aw)^n-(aw)(xw)(aw)^n||^{\frac{1}{n}}=0.$$

Therefore $a\in \mathcal{A}^{\tiny\textcircled{d},w}$ and $$a^{\tiny\textcircled{d},w}=x=[a^{d,w}w]^2(a^{d,w})^{\tiny\textcircled{\#},w}.$$\end{proof}

As an immediate consequence, we provide formulas of the pseudo weighted core inverse of a complex matrix.

\begin{cor} Let $A,W\in {\Bbb C}^{n\times n}$. Then $$A^{\tiny\textcircled{\dag},W}=[A^{D,W}W]^2(A^{D,W})^{\tiny\textcircled{\#},W}.$$
 \end{cor}
\begin{proof} This is obvious by Theorem 4.2.\end{proof}

If $a,x$ and $w$ satisfy the equations $a=awxwa$ and $(wawx)^*=wawx$, then $x$ is called weighted $(1,3,w)$-inverse of $a$ and is denoted by $a^{(1,3)_w}$. We use $\mathcal{A}^{(1,3)_w}$ to stand for sets of all weighted $(1,3,w)$-invertible elements in $\mathcal{A}$.

\begin{thm} Let $a\in \mathcal{A}$. Then the following are equivalent:\end{thm}
\begin{enumerate}
\item [(1)] $a\in \mathcal{A}^{\tiny\textcircled{d},w}$.
\vspace{-.5mm}
\item [(2)] $a\in \mathcal{A}^{d,w}$ and $a^{d,w}\in \mathcal{A}^{(1,3)_w}$.
\end{enumerate}
In this case, $a^{\tiny\textcircled{d},w}=(a^{d,w})^2(a^{d,w})^{(1,3)_w}.$
\begin{proof}  $(1)\Rightarrow (2)$ In view of Lemma 4.1, $a\in \mathcal{A}^{d,w}$. Let $x=a^{\tiny\textcircled{d},w}$. Then we have
$xwawx=x$ and and $(wawx)^*=wawx$. Let $z=(aw)^2x$. Then we verify that
$$\begin{array}{rll}
wa^{d,w}wz&=&wa^{d,w}w(aw)^2x=w[(aw)^d]^2aw(aw)^2x\\
&=&w(aw)^d(aw)^2x=wawx,\\
(wa^{d,w}wz)^*&=&(wawx)^*=wawx=wa^{d,w}wz,\\
a^{d,w}wzwa^{d,w}&=&(aw)^d(aw)^2xw[(aw)^d]^2a=(aw)^da[(waw)x(waw)][(aw)^d]^3a\\
&=&(aw)^da\big(w[(aw)^d]^2a\big)=(aw)^da(wa)^d=a^{d,w}.
\end{array}$$
Therefore $a^{d,w}\in \mathcal{A}^{(1,3)_w}$, as desired.

$(2)\Rightarrow (1)$ Let $x=(a^{d,w}w)^2(a^{d,w})^{(1,3)_w}$. Then we check that
$$\begin{array}{rll}
xwawx&=&(a^{d,w}w)^2(a^{d,w})^{(1,3)_w}waw(a^{d,w}w)^2(a^{d,w})^{(1,3)_w}\\
&=&(a^{d,w}w)[a^{d,w}w(a^{d,w})^{(1,3)_w}wa^{d,w}]w(a^{d,w})^{(1,3)_w}\\
&=&(a^{d,w}w)a^{d,w}w(a^{d,w})^{(1,3)_w}=(a^{d,w}w)^2(a^{d,w})^{(1,3)_w}=x.
\end{array}$$

Claim 1. $xw\mathcal{A}=(aw)^d\mathcal{A}$. Obviously, $xw\mathcal{A}\subseteq a^{d,w}\mathcal{A}\subseteq (aw)^d\mathcal{A}$.
On the other hand, $a^{d,w}=[(a^{d,w}w)^2(a^{d,w})^{(1,3)_w}]wa^{d,w}wa=(xw)aa^{d,w}wa$, and then
$(aw)^d\mathcal{A}\subseteq a^{d,w}\mathcal{A}\subseteq xw\mathcal{A}$. We infers that
$xw\mathcal{A}=(aw)^d\mathcal{A}$.

Claim 2. $(wx)^*\mathcal{A}=(wa)^d\mathcal{A}$. Clearly, we have $wx=w(a^{d,w}w)^2(a^{d,w})^{(1,3)_w}=wa^{d,w}[wa^{d,w}w(a^{d,w})^{(1,3)_w}]$.
Hence, $$\begin{array}{rll}
(wx)^*&=&[wa^{d,w}w(a^{d,w})^{(1,3)_w}]^*[wa^{d,w}]^*\\
&=&wa^{d,w}w(a^{d,w})^{(1,3)_w}[wa^{d,w}]^*\\
&=&w[(aw)^d]^2aw(a^{d,w})^{(1,3)_w}[wa^{d,w}]^*(wa)^dw(a^{d,w})^{(1,3)_w}[wa^{d,w}]^*.
\end{array}$$ Hence, $(wx)^*\mathcal{A}\subseteq (wa)^d\mathcal{A}$. On the other hand,
$$\begin{array}{rll}
(wa)^d&=&wa[(wa)^d]^2=wa^{d,w}=wa^{d,w}w(a^{d,w})^{(1,3)_w}wa^{d,w}\\
&=&[wa^{d,w}w(a^{d,w})^{(1,3)_w}]^*wa^{d,w}=[waw(a^{d,w}w)^2(a^{d,w})^{(1,3)_w}]^*wa^{d,w}\\
&=&(wx)^*(wa)^*wa^{d,w};
\end{array}$$ hence,
$(wa)^d\mathcal{A}\subseteq (wx)^*\mathcal{A}$. Thus, $(wx)^*\mathcal{A}=(wa)^d\mathcal{A}$.
Therefore $a\in \mathcal{A}^{\tiny\textcircled{d},w}$ by ~\cite[Theorem 3.1]{C}.\end{proof}

\begin{cor} Let $A,W\in {\Bbb C}^{n\times n}$. Then $A^{\tiny\textcircled{\dag},W}=(A^{D,W})^2(A^{D,W})^{(1,3)_W}.$
\end{cor}
\begin{proof} This is obvious by Theorem 4.4.\end{proof}

\begin{lem} Let $a\in \mathcal{A}^{\tiny\textcircled{d}}$ and $b\in \mathcal{A}$. Then the following are equivalent:\end{lem}
\begin{enumerate}
\item [(1)] $(1-wawa^{\tiny\textcircled{d},w})b=0$.
\vspace{-.5mm}
\item [(2)] $(1-wa^{\tiny\textcircled{d},w}wa)b=0$.
\vspace{-.5mm}
\item [(3)] $(wa)^{\pi}b=0$.
\end{enumerate}
\begin{proof} $(1)\Rightarrow (2)$ Since $(1-wawa^{\tiny\textcircled{d},w})b=0$, we have $b=wawa^{\tiny\textcircled{d},w}b$. Hence,
$$\begin{array}{rll}
(1-wa^{\tiny\textcircled{d}}wa)b&=&[1-wa^{\tiny\textcircled{d}}wa]wawa^{\tiny\textcircled{d},w}b\\
&=&w[awa^{\tiny\textcircled{d},w}-(a^{\tiny\textcircled{d}}w)(aw)^2a^{\tiny\textcircled{d},w}]b\\
&=&w[awa^{\tiny\textcircled{d},w}-(aw)a^{\tiny\textcircled{d},w}]b=0,
\end{array}$$ as required.

$(2)\Rightarrow (3)$ Since $(1-wa^{\tiny\textcircled{d},w}wa)b=0$, we have $b=wa^{\tiny\textcircled{d},w}wab$.
In view of Theorem 4.2, we have $a^{\tiny\textcircled{d},w}=(a^{d,w})^2(a^{d,w})^{\tiny\textcircled{\#},w}$. Thus,
$$\begin{array}{rll}
(1-wa(wa)^d)b&=&(1-wa(wa)^d)wa^{\tiny\textcircled{d},w}wab\\
&=&(1-wa(wa)^d)w(a^{d,w})^2(a^{d,w})^{\tiny\textcircled{\#},w}wab\\
&=&(1-wa(wa)^d)w[(aw)^d]^2aa^{d,w}(a^{d,w})^{\tiny\textcircled{\#},w}wab\\
&=&(1-wa(wa)^d)(wa)^da^{d,w}(a^{d,w})^{\tiny\textcircled{\#},w}wab\\
&=&0.
\end{array}$$

$(3)\Rightarrow (1)$ By hypothesis, we have $b=wa(wa)^db$. Then
$$\begin{array}{rll}
(1-wawa^{\tiny\textcircled{d},w})b&=&(1-wawa^{\tiny\textcircled{d},w})wa(wa)^db\\
&=&w[aw-(aw)(a^{\tiny\textcircled{d},w}w)(aw)]a[(wa)^d]^2b\\
&=&0,
\end{array}$$ as asserted.\end{proof}

We come now to present the weighted generalized core-EP inverse of a triangular matrix over a Banach algebra.

\begin{thm} Let $\mathcal{A}$ be a Banach algebra, $W=wI_2$ and $M=\left(
\begin{array}{cc}
a&b\\
0&d
\end{array}
\right)$ with $a,d\in \mathcal{A}^{\tiny\textcircled{d},w}$. If $(wa)^{\pi}b=0$,
then $M\in M_2(\mathcal{A})^{\tiny\textcircled{d},W}$ and $$x^{\tiny\textcircled{d},W}=\left(
\begin{array}{cc}
a^{\tiny\textcircled{d},w}&-a^{\tiny\textcircled{d},w}wbwd^{\tiny\textcircled{d},w}\\
0&d^{\tiny\textcircled{d},w}
\end{array}
\right).$$\end{thm}
\begin{proof} By hypothesis, we have $w$-weighted generalized core-EP decompositions:
$$a=x+y, d=s+t,$$ where $$x,s\in \mathcal{A}^{\tiny\textcircled{\#},w}, y,t\in \mathcal{A}^{qnil}$$ and $$ywx=0, tws=0.$$
As in the proof of Theorem 3.1, $$\begin{array}{rll}
x&=&awa^{\tiny\textcircled{d},w}wa,y=a-awa^{\tiny\textcircled{d},w}wa;\\
s&=&dwd^{\tiny\textcircled{d},w}wd,t=d-dwd^{\tiny\textcircled{d},w}wd.
\end{array}$$ Then we have $M=P+Q$, where
$$P=\left(
\begin{array}{cc}
x&b\\
0&s
\end{array}
\right), Q=\left(
\begin{array}{cc}
y&0\\
0&t
\end{array}
\right).$$

Set $X=\left(
\begin{array}{cc}
x^{\tiny\textcircled{\#},w}&-x^{\tiny\textcircled{\#},w}wbws^{\tiny\textcircled{\#},w}\\
0&s^{\tiny\textcircled{\#},w}
\end{array}
\right)$. Then we verify that

Set $W=wI_2$.

$$\begin{array}{rll}
PWX&=&\left(
\begin{array}{cc}
xw&bw\\
0&sw
\end{array}
\right)\left(
\begin{array}{cc}
x^{\tiny\textcircled{\#},w}&-x^{\tiny\textcircled{\#},w}wbws^{\tiny\textcircled{\#},w}\\
0&s^{\tiny\textcircled{\#},w}
\end{array}
\right)\\
&=&\left(
\begin{array}{cc}
xwx^{\tiny\textcircled{d},w}&0\\
0&sws^{\tiny\textcircled{d},w}
\end{array}
\right)\\
P(WX)^2&=&\left(
\begin{array}{cc}
xwx^{\tiny\textcircled{d},w}w&0\\
0&sws^{\tiny\textcircled{d},w}w
\end{array}
\right)X=X,\\
(WPWX)^*&=&\left(
\begin{array}{cc}
wxwx^{\tiny\textcircled{d},w}&0\\
0&wsws^{\tiny\textcircled{d},w}
\end{array}
\right)^*=WPWX,\\
WPWXWPW&=&\left(
\begin{array}{cc}
wxwx^{\tiny\textcircled{d},w}&0\\
0&wsws^{\tiny\textcircled{d},w}
\end{array}
\right)\left(
\begin{array}{cc}
wxw&wbw\\
0&wsw
\end{array}
\right)\\
&=&\left(
\begin{array}{cc}
wxwx^{\tiny\textcircled{d},w}wxw&wxwx^{\tiny\textcircled{d},w}wbw\\
0&wsws^{\tiny\textcircled{d},w}wsw
\end{array}
\right)=WPW.
\end{array}$$

Thus, $P$ has $W$-weighted core inverse and $P^{\tiny\textcircled{\#},W}=X$. Since $y,t\in \mathcal{A}_w^{qnil}$, we directly verify that $Q\in M_2(\mathcal{A})_{W}^{qnil}$. Clearly, $ywb=(a-awa^{\tiny\textcircled{d},w}wa)wb=a[(1-wa^{\tiny\textcircled{d},w}wa)wb]=0.$
Then we verify that
$$\begin{array}{rll}
QWP&=&\left(
\begin{array}{cc}
yw&0\\
0&tw
\end{array}
\right)\left(
\begin{array}{cc}
x&b\\
0&s
\end{array}
\right)=\left(
\begin{array}{cc}
0&ywb\\
0&0
\end{array}
\right)=0.
\end{array}$$
In light of Theorem 3.1, we derive that
$$M^{\tiny\textcircled{d},W}=P^{\tiny\textcircled{\#},W}=\left(
\begin{array}{cc}
x^{\tiny\textcircled{\#},w}&-x^{\tiny\textcircled{\#},w}wbws^{\tiny\textcircled{\#},w}\\
0&s^{\tiny\textcircled{\#},w}
\end{array}
\right).$$
Therefore $$M^{\tiny\textcircled{d},W}=\left(
\begin{array}{cc}
a^{\tiny\textcircled{d},w}&-a^{\tiny\textcircled{d},w}wbwd^{\tiny\textcircled{d},w}\\
0&d^{\tiny\textcircled{d},w}
\end{array}
\right).$$\end{proof}

\begin{cor} Let $\mathcal{A}$ be a Banach algebra, $W=wI_2$ and $M=\left(
\begin{array}{cc}
a&0\\
c&d
\end{array}
\right)$ with $a,d\in \mathcal{A}^{\tiny\textcircled{d},w}$. If $(wd)^{\pi}c=0$, then $M\in M_2(\mathcal{A})^{\tiny\textcircled{d},W}$ and $$x^{\tiny\textcircled{d},W}=\left(
\begin{array}{cc}
a^{\tiny\textcircled{d},w}&0\\
-d^{\tiny\textcircled{d},w}wcwa^{\tiny\textcircled{d},w}&d^{\tiny\textcircled{d},w}
\end{array}
\right).$$\end{cor}
\begin{proof} Clearly, we have $\left(
  \begin{array}{cc}
   0&1\\
  1&0
  \end{array}
\right)M\left(
  \begin{array}{cc}
   0&1\\
  1&0
  \end{array}
\right)=\left(
  \begin{array}{cc}
   d&c\\
   0&a
  \end{array}
\right)$. Applying Theorem 4.7 to the matrix $\left(
  \begin{array}{cc}
   d&c\\
   0&a
  \end{array}
\right)$, we see that $\left(
  \begin{array}{cc}
   d&c\\
   0&a
  \end{array}
\right)$ has generalized $W$-core-EP inverse. This implies that $M$ has generalized $W$-core-EP inverse.
In this case, $$M^{\tiny\textcircled{d},W}=\left(
  \begin{array}{cc}
   0&1\\
  1&0
  \end{array}
\right)\left(
  \begin{array}{cc}
   d&c\\
   0&a
  \end{array}
\right)^{\tiny\textcircled{d},W}\left(
  \begin{array}{cc}
   0&1\\
  1&0
  \end{array}
\right).$$ Therefore we complete the proof by Theorem 4.7.\end{proof}

\begin{cor} Let $\mathcal{A}$ be a Banach algebra and $M=\left(
\begin{array}{cc}
a&b\\
0&d
\end{array}
\right)$ with $a,d\in \mathcal{A}^{\tiny\textcircled{d}}$. If $a^{\pi}b=0$, then $M\in M_2(\mathcal{A})^{\tiny\textcircled{d}}$ and $$x^{\tiny\textcircled{d}}=\left(
\begin{array}{cc}
a^{\tiny\textcircled{d}}&-a^{\tiny\textcircled{d}}bd^{\tiny\textcircled{d}}\\
0&d^{\tiny\textcircled{d}}
\end{array}
\right).$$\end{cor}
\begin{proof} Straightforward by choosing $w=1$ in Theorem 4.7.\end{proof}

{\bf Conflict of interest}

No potential conflict of interest was reported by the authors.\\

{\bf Data Availability Statement}

No/Not applicable (this manuscript does not report data generation or analysis).

\end{document}